\documentclass{amsproc}

\usepackage{amssymb}
\usepackage{graphicx}
\usepackage{amscd}
\usepackage{amsmath}
\theoremstyle{plain}

\newtheorem{definition}{Definition}

\newtheorem{lemma}{Lemma}

\newtheorem{proposition}{Proposition}
\newtheorem{procedure}{Procedure}
\newtheorem{remark}{Remark}

\newtheorem{theorem}{Theorem}

\numberwithin{equation}{section}

\begin{document}

\title[On the complex moment problem]{On the complex moment problem as a dynamic inverse problem for
a discrete system}

\author{A. S. Mikhaylov}
\address{St. Petersburg   Department   of   V.A. Steklov    Institute   of   Mathematics
of   the   Russian   Academy   of   Sciences, 7, Fontanka, 191023
St. Petersburg, Russia and Saint Petersburg State University,
St.Petersburg State University, 7/9 Universitetskaya nab., St.
Petersburg, 199034 Russia.}
\email{mikhaylov@pdmi.ras.ru}

\author{ V. S. Mikhaylov}
\address{St.Petersburg   Department   of   V.A.Steklov    Institute   of   Mathematics
of   the   Russian   Academy   of   Sciences, 7, Fontanka, 191023
St. Petersburg, Russia.} \email{ftvsm78@gmail.com}

\thanks{No new data were created or analyzed during this study. Data sharing is not applicable to this article.}

\thanks{All authors declare that they have no conflicts of interest.}

\date{September, 2025}

\maketitle






\noindent {\bf Abstract.} We consider the complex moment problem,
that is the problem of constructing a positive Borel measure on
$\mathbb{C}$ from a given set of moments. We relate this problem
to the dynamic inverse problem for the discrete system associated
with the complex Jacobi matrix. We propose a method that allows
one to construct a discrete measure which is a solution to the
truncated moment problem, we also show how the characterization of
dynamic inverse data in solving the inverse problem provides
sufficient conditions for solving the full complex moment problem.

\section{Introduction.}

The complex moment problem is: given a set of complex numbers
$s_0,s_1,s_2,\ldots$, to find a Borel measure $d\rho$ on
$\mathbb{C}$ such that
\begin{equation}
\label{Moment_eq} s_k=\int_{\mathbb{C}}
\lambda^k\,d\rho(\lambda),\quad k=0,1,2,\ldots.
\end{equation}
If such a measure exists $s_0,s_1,s_2,\ldots$ are called moments
of this measure. The truncated moment problem is the problem of
finding a measure that satisfies a finite number of moment
equalities (\ref{Moment_eq}) for $k=0,1,\ldots,2N-2$ for some
$N\in \mathbb{N}$.

For a given sequence of complex numbers $\{a_1,a_2,\ldots\}$,
$\{b_1, b_2,\ldots  \}$, $a_i\not= 0$, we define the complex
Jacobi matrix:
\begin{equation} \label{Jac_matr}
A=\begin{pmatrix} b_1 & a_1 & 0 & 0 & 0 &\ldots \\
a_1 & b_2 & a_2 & 0 & 0 &\ldots \\
0 & a_2 & b_3 & a_3 & 0 & \ldots \\
\ldots &\ldots  &\ldots &\ldots & \ldots &\ldots
\end{pmatrix}.
\end{equation}
For $N\in \mathbb{N}$, by $A^N$ we denote the $N\times N$ Jacobi
matrix which is a block of (\ref{Jac_matr}) consisting of the
intersection of first $N$ columns with first $N$ rows of $A$.

Associated with the matrix $A$ and additional parameter
$\mathbb{C}\ni a_0\not=0$ is a dynamical system with discrete
time:
\begin{equation}
\label{Jacobi_dyn}
\begin{cases}
u_{n,t+1}+u_{n,t-1}-a_{n}u_{n+1,t}-a_{n-1}u_{n-1,t}-b_nu_{n,t}=0,\quad n,t\in \mathbb{N},\\
u_{n,-1}=u_{n,0}=0,\quad n\in \mathbb{N}, \\
u_{0,t}=f_t,\quad t\in \mathbb{N}\cup\{0\},
\end{cases}
\end{equation}
which is a discrete analogue of the dynamical systems governed by
a wave equation on the semi-axis \cite{AM,BM_1}. By analogy with
the continuous problems \cite{B07}, we consider the complex
sequence $f=(f_0,f_1,\ldots)$ as a \emph{boundary control}. The
solution to (\ref{Jacobi_dyn}) we denote by $u^f_{n,t}$. We also
consider the dynamical system associated with finite matrix $A_N$:
\begin{equation}
\label{Jacobi_dyn_int}
\begin{cases}
v_{n,t+1}+v_{n,t-1}-a_nv_{n+1,t}-a_{n-1}v_{n-1,t}-b_nv_{n,t}=0,\,\, t\in \mathbb{N}_0,\,\, n\in 1,\ldots, N,\\
v_{n,\,-1}=v_{n,\,0}=0,\quad n=1,2,\ldots,N+1, \\
v_{0,\,t}=f_t,\quad v_{N+1,\,t}=0,\quad t\in \mathbb{N}_0,
\end{cases}
\end{equation}
which is a natural analog of the dynamical systems governed by a
wave equation on the interval, the solution to
(\ref{Jacobi_dyn_int}) is denoted by $v^f$.

Fixing $T\in \mathbb{N}$, we associate the \emph{response
operator} with (\ref{Jacobi_dyn}), which maps the control
$f=(f_0,\ldots f_{T-1})$ to $u^f_{1,t}$:
\begin{equation}
\label{Resp_op} \left(R^T f\right)_t:=u^f_{1,t},\quad t=1,\ldots,
T.
\end{equation}

In the second section using the Autonne-Takagi \cite{W}
factorization, we derive a special "spectral representation" for
the solution of (\ref{Jacobi_dyn_int}).

In the third section we present results on the dynamic inverse
problem (IP) for the system (\ref{Jacobi_dyn}) in accordance with
\cite{MM4}. This problems is a natural discrete analogue of the IP
for the wave equation on the half-axis, where the dynamic
Dirichlet to Neumann map is used as inverse data, see \cite{B07}.
The IP for the dynamical system (\ref{Jacobi_dyn}) with real
Jacobi matrix is considered in \cite{MM1,MM2}. The connections
between the IP for the system with real Jacobi matrix and
classical moment problems are described in \cite{MM3,MM5}. The IP
for the dynamical system with complex Jacobi matrix is considered
in \cite{MM4}.

In last section we describe connections between the IPs for
(\ref{Jacobi_dyn_int}) and (\ref{Jacobi_dyn}) and complex moment
problem and introduce a discrete Borel measure on $\mathbb{C}$,
concentrated on the finite set of point, associated with
(\ref{Jacobi_dyn_int}), which is a solution to the truncated
complex moment problem. We emphasize that we propose a method that
allows one to construct a discrete measure, i.e. find points of
its support and determine the masses at these points. Letting $N$
tend to infinity in (\ref{Jacobi_dyn_int}) we obtain a sequence of
measures that converges to a solution to the full complex moment
problem.

Another approach to the complex moment problem was proposed in
\cite{Zag}, where the author used the generalized spectral
function introduced in \cite{G} as a basic tool and obtained
sufficient conditions on the moment sequence under the condition
that it is bounded. Note that our method do not require the
boundedness of $s_k$, $k\geqslant 0$.

If $H\in C^{n\times n}$ then by $H^*$ we denote the matrix
conjugate to $H$ and recall that $H^*=\overline{H^\top}$, where
$H^\top$ is a transpose of $H$.

\section{Discrete dynamical system. Forward problem, special representation of the solution.}

We first derive a special representation (Fourier-type expansion)
of the the solution to (\ref{Jacobi_dyn_int}). Note that another
representation (Duhamel-type), important for solving the IP
obtained in \cite{MM4}, is used in the next section. Our special
representation is based on the following Autonne-Takagi \cite{W}
factorization:
\begin{theorem}
Let $H\in \mathbb{C}^{n\times n}$ be a complex symmetric matrix:
$H^*=\overline H$, then there exists a unitary matrix $U$ such
that
\begin{equation}
\label{AutTak}
UHU^\top=D=\begin{pmatrix} \hat{d}_1& 0&\ldots &0\\
0& \hat{d}_2&\ldots &0\\
 \ldots& \ldots&\ldots &\ldots\\
 0& 0&\ldots &\hat{d}_n
\end{pmatrix},
\end{equation}
where $\hat{d}_i\geqslant 0,$ $i=1\ldots,n$.
\end{theorem}

We apply this theorem for the matrix $A^N$: there exist unitary
$U^N$ such that $U^NA^N\left(U^N\right)^\top$ as in the theorem.
Note that in our case all $\hat{d}_i>0$, since otherwise $A^Nu=0$
implies $\det A^N=0$, which in turn implies linear dependence of
some rows or columns of $A^N$, which is impossible. At the same
time some of the $\hat{d}_i$ can coincide. In the latter case we
modify the unitary matrix as follows: assume for example that
$\hat{d}_1=\hat{d}_2$, then, taking
\begin{equation*}
U_1^N:=\begin{pmatrix} 1& 0&\ldots &0\\
0& e^{i\frac{\varphi_1}{2}}&\ldots &0\\
 \ldots& \ldots&\ldots &\ldots\\
 0& 0&\ldots &1
\end{pmatrix}U^N,
\end{equation*}
we obtain that
\begin{equation*}
U_1^NA^N\left(U_1^N\right)^\top=\begin{pmatrix} \hat{d}_1& 0&\ldots &0\\
0& \hat{d}_2e^{i\varphi}&\ldots &0\\
 \ldots& \ldots&\ldots &\ldots\\
 0& 0&\ldots &\hat{d}_n
\end{pmatrix}.
\end{equation*}

Using these arguments we can make all $\hat{d}_i$ in the
representation (\ref{AutTak}) distinct (but some of them can be
complex). To summarize what we did so far, in what follows we
choose unitary $U$ (we drop $N$) such that
\begin{equation}
\label{Factor}
UA^NU^\top=\begin{pmatrix} \hat{d}_1& 0&\ldots &0\\
0& \hat{d}_2&\ldots &0\\
 \ldots& \ldots&\ldots &\ldots\\
 0& 0&\ldots &\hat{d}_n
\end{pmatrix},\quad \hat{d}_i\in \mathbb{C}\backslash\{0\}, \hat{d}_i\not=\hat{d}_j,\, i\not=j,\,i,j=1\ldots,N.
\end{equation}

We introduce the vectors
\begin{equation*}
e_i=\begin{pmatrix}0\\ \cdot\\1\\ \cdot\\0\end{pmatrix},\quad
\text{with $1$ on i-th place}, i=1,\ldots,N.
\end{equation*}
Then
\begin{equation*}
UA^NU^\top e_i=\hat{d}_ie_i,
\end{equation*}
we multiply the equality above from the left by $U^*$ and get
\begin{equation*}
A^NU^\top e_i=\hat{d}_iU^*e_i=\hat{d}_i\overline{U^\top}e_i,
\end{equation*}
Introducing the notation
\begin{equation*}
U^\top =\left(\hat{u}^1 \,|\, \hat{u}^2 \,|\, \ldots \,|\,
\hat{u}^N\right),\quad \hat{u}^i=\begin{pmatrix} \hat{u}^i_1\\
{\hat{u}}^i_2\\
\ldots\\
{\hat{u}}^i_n
\end{pmatrix}
\end{equation*}
that is $\hat{u}^i$ is a column in the matrix $U^\top$, we see
that $\hat{u}^i$ satisfies:
\begin{equation}
\label{Spectral_strange}
A^N\hat{u}^i=\hat{d}_i\overline{\hat{u}^i},\quad A^N\begin{pmatrix} \hat{u}^i_1\\
{\hat{u}}^i_2\\
\ldots\\
{\hat{u}}^i_n
\end{pmatrix}=\hat{d}_i \begin{pmatrix}\overline{\hat{u}^i_1}\\
\overline{{u}^i_2}\\
\ldots\\
\overline{\hat{u}^i_n}
\end{pmatrix}.
\end{equation}
Note that the first components of all vectors are non-zero:
\begin{equation*}
\hat{u}^i_1\not=0,\quad i=1,\ldots,N,
\end{equation*}
otherwise it immediately follows from (\ref{Spectral_strange})
that $\hat{u}^i=0$.

Thus, we can introduce the vectors which we use in the
Fourier-type expansion for the solution to (\ref{Jacobi_dyn_int})
in the following way:
\begin{equation*}
u^i=\begin{pmatrix} 0\\
\frac{\hat{u}^i_1}{\hat{u}^i_1}\\
\ldots\\
\frac{\hat{u}^i_N}{\hat{u}^i_1}\\
0
\end{pmatrix},\quad i=1,\ldots,N,
\end{equation*}
so we formally add two values: $u^i_0=u^i_{N+1}=0$ and normalize
vectors such that $u^i_1=1$. In this case we have (see
(\ref{Spectral_strange}))
\begin{equation*}
A^N\begin{pmatrix} \frac{\hat{u}^i_1}{\hat{u}^i_1}\\
\frac{\hat{u}^i_2}{\hat{u}^i_1}\\
\ldots\\
\frac{\hat{u}^i_n}{\hat{u}^i_1}
\end{pmatrix}={d}_i \begin{pmatrix}\frac{\overline{\hat{u}^i_1}}{\overline{\hat{u}^i_1}}\\
\frac{\overline{\hat{u}^i_2}}{\overline{\hat{u}^i_1}}\\
\ldots\\
\frac{\overline{\hat{u}^i_n}}{{\overline{\hat{u}^i_1}}}
\end{pmatrix},
\end{equation*}
where we introduced the notation
\begin{equation}
\label{EigVal}
d_i:=\hat{d_i}\frac{\overline{\hat{u}^i_1}}{{\hat{u}^i_1}},\quad
i=1,\ldots,N.
\end{equation}

We take some vector $(y_0,y_1,\ldots y_N,y_{N+1})$, multiply the
equation in (\ref{Jacobi_dyn_int}) by $y_n$, sum up and change the
order of summation:
\begin{eqnarray*}
0=\sum_{n=1}^N
\left(v_{n,t+1}y_n+v_{n,t-1}y_n-a_{n-1}v_{n,t}y_{n-1}-a_nv_{n,t}y_{n+1}-b_nv_{n,t}y_n\right)\\-
a_Nv_{N+1,t}y_N-a_0v_{0,t}y_1+a_0v_{1,t}y_0+a_Nv_{N,t}y_{N+1}.
\end{eqnarray*}
Now we take  $\left(y_0,y_1,\ldots
y_N,y_{N+1}\right)=\left(u^i_0,u^i_1,\ldots
u^i_N,u^i_{N+1}\right)$ and evaluate counting  $u^i_1=1$ and the
values at $n=0$ and $n=N+1$:
\begin{eqnarray*}
0=\sum_{n=1}^N
\left(v_{n,t+1}u^i_n+v_{n,t-1}u^i_n-v_{n,t}\left(a_{n-1}y^i_{n-1}+a_nvu^i_{n+1}+b_nu^i_n\right)\right)
-a_0v_{0,t}u^i_1.
\end{eqnarray*}
Counting (\ref{Spectral_strange}) we get:
\begin{equation}
\label{Repr1} 0=\sum_{n=1}^N
\left(v_{n,t+1}u^i_n+v_{n,t-1}u^i_n-d_iv_{n,t}\overline{u^i_n}\right)
-a_0f_t\quad i=1,\ldots,N.
\end{equation}

Now we look for the solution to (\ref{Jacobi_dyn_int}) in the
form:
\begin{equation}
\label{Repr2}
v_{n,t}=\begin{cases} \sum_{k=1}^N c^k_t\overline {u^k_n},\\
f_t,\quad n=0.
\end{cases}
\end{equation}
Plugging (\ref{Repr2}) into (\ref{Repr1}) we have:
\begin{equation}
\label{Repr3} \sum_{n=1}^N \left(\sum_{k=1}^N
\left(c_{t+1}^k\overline{u^k_n}+c_{t-1}^k\overline{u^k_n}\right)u_n^i-d_i\sum_{k=1}^Nc^k_t\overline{u^k_n}\overline{u^i_n}\right)
=a_0f_t.
\end{equation}
Introducing the notations
\begin{eqnarray}
\sum_{n=1}^N \overline{u^k_n}u_n^i=\delta_{ki}\rho_i,\quad i=1,\ldots,N,\label{Rho}\\
H_{ki}=\sum_{n=1}^N \overline{u^k_n}\overline{u_n^i}, \quad
k,i=1,\ldots,N,\label{HKL}
\end{eqnarray}
we rewrite (\ref{Repr3}) as:
\begin{equation*}
\left(c_{t+1}^k+c_{t-1}^k\right)\delta_{ki}\rho_i-d_i\sum_{k=1}^N
c^k_tH_{ki}=a_0f_t,\quad t,k,i=1\ldots,N.
\end{equation*}
So we see that $c^i_t$ are determined from
\begin{equation}
\label{Repr4} c_{t+1}^i+c_{t-1}^i -\frac{d_i}{\rho_i}\sum_{k=1}^N
c^k_tH_{ki}=\frac{a_0}{\rho_i}f_t,\quad t,i=1,\ldots,N.
\end{equation}
We look for the solution to (\ref{Repr4}) in the form:
\begin{equation}
\label{CT_coeff}
c_t^i=\frac{a_0}{\rho_i}\sum_{l=0}^tf_lT_{t-l}^{(i)}, \quad
t,i=1,\ldots,N.
\end{equation}
Plugging this representation into (\ref{Repr4}) we get:
\begin{equation*}
\frac{a_0}{\rho_i}\left(\sum_{l=0}^tf_lT_{t+1-l}^{(i)}+\sum_{l=0}^{t-1}f_lT_{t-1-l}^{(i)}-\sum_{k=1}^N
d_i\frac{H_{ki}}{\rho_k}\sum_{l=0}^tf_lT_{t-l}^{(i)}\right)=\frac{a_0}{\rho_i}f_t,
\end{equation*}
changing the order of summation (at this point we use the
additional value $T_{-1}^{(i)}$) we come to
\begin{equation}
\label{Repr5} \sum_{l=0}^t
f_l\left(T_{t+1-l}^{(i)}+T_{t-1-l}^{(i)}-\left(\sum_{k=1}^N
d_i\frac{H_{ki}}{\rho_k}\right)T_{t-l}^{(i)}\right)+f_{t+1}T_0^{(i)}-f_tT_{-1}^{(i)}=f_t.
\end{equation}
We introduce the notation
\begin{equation}
\label{Omega} \omega_i=\sum_{k=1}^N d_i\frac{H_{ki}}{\rho_k},\quad
i=1,\ldots,N,
\end{equation}
then (\ref{Repr5}) holds if $T^{(i)}_t$ satisfies
\begin{equation}
\label{T_coeff}
\begin{cases}
T_{t+1}^{(i)}+T_{t-1}^{(i)}-\omega_iT_t^{(i)}=0,\quad t=0,1\ldots,N,\\
T_0^{(i)}=0,\,\, T_{-1}^{(i)}=-1.
\end{cases}
\end{equation}
Or, in other words, $T_{t}^{(i)}$ are simply the Chebyshev
polynomials of the first kind evaluated at points $\omega_i$:
$T^{(i)}_t=T_t(\omega_i)$.

\section{Inverse problem for discrete dynamical system associated with complex Jacobi matrices.}

In this section we outline the results of the IP for
(\ref{Jacobi_dyn}) according to \cite{MM4}.

We fix some positive integer $T$ and denote by $\mathcal{F}^T$ the
\emph{outer space} of the system (\ref{Jacobi_dyn}), the space of
controls: $\mathcal{F}^T:=\mathbb{C}^T$, $f\in \mathcal{F}^T$,
$f=(f_0,\ldots,f_{T-1})$, $f,g\in \mathcal{F}^T$,
$(f,g)_{\mathcal{F}^T}=\sum_{k=0}^{T-1} f_k\overline{g_k}$. And
let $\mathcal{F}^\infty=\left\{(f_0,f_1,\ldots)\,|\, f_i\in
\mathbb{C},\, i=0,1,\ldots  \right\}$, so $\mathcal{F}^\infty$ is
the set of complex sequences. The following representation formula
for the solution to (\ref{Jacobi_dyn}) can be considered as an
analogue of a Duhamel representation formula for the
initial-boundary value problem for the wave equation with a
potential on the half-line \cite{AM}.
\begin{lemma}
A solution to (\ref{Jacobi_dyn}) admits the representation
\begin{equation}
\label{Jac_sol_rep} u^f_{n,t}=\prod_{k=0}^{n-1}
a_kf_{t-n}+\sum_{s=n}^{t-1}w_{n,s}f_{t-s-1},\quad n,t\in
\mathbb{N},
\end{equation}
where $w_{n,s}$ satisfies the Goursat problem
\begin{equation*}
\left\{
\begin{array}l
w_{n,s+1}+w_{n,s-1}-a_nw_{n+1,s}-a_{n-1}w_{n-1,s}-b_nw_{n,s}=\\
=-\delta_{s,n}(1-a_n^2)\prod_{k=0}^{n-1}a_k,\,n,s\in \mathbb{N}, \,\,s>n,\\
w_{n,n}-b_n\prod_{k=0}^{n-1}a_k-a_{n-1}w_{n-1,n-1}=0,\quad n\in \mathbb{N},\\
w_{0,t}=0,\quad t\in \mathbb{N}_0.
\end{array}
\right.
\end{equation*}
\end{lemma}

\begin{definition}
For $f,g\in \mathcal{F}^\infty$ we define the convolution
$c=f*g\in \mathcal{F}^\infty$ by the formula
\begin{equation*}
c_t=\sum_{s=0}^{t}f_sg_{t-s},\quad t\in \mathbb{N}\cup \{0\}.
\end{equation*}
\end{definition}

Let us introduce an analog of the dynamic \emph{response operator}
(dynamic Dirichlet-to-Neumann map) \cite{B07} for the system
(\ref{Jacobi_dyn}):
\begin{definition}
The \emph{response operator} $R^T:\mathcal{F}^T\mapsto
\mathbb{C}^T$ for the system (\ref{Jacobi_dyn}) is defined by
(\ref{Resp_op})
\end{definition}
The \emph{response vector} is the convolution kernel of the
response operator,
$r=(r_0,r_1,\ldots,r_{T-1})=(a_0,w_{1,1},w_{1,2},\ldots
w_{1,T-1})$, according to (\ref{Jac_sol_rep}):
\begin{eqnarray}
\label{R_def}
\left(R^Tf\right)_t=u^f_{1,t}=a_0f_{t-1}+\sum_{s=1}^{t-1}
w_{1,s}f_{t-1-s}
\quad t=1,\ldots,T.\\
\notag \left(R^Tf\right)=r*f_{\cdot-1}.
\end{eqnarray}
 By choosing the special control
 $f=\delta=(1,0,0,\ldots)$, the kernel of the response operator can be determined as
\begin{equation}
\label{con11} \left(R^T\delta\right)_t=u^\delta_{1,t}=
r_{t-1},\quad t=1,2,\ldots.
\end{equation}
The inverse problem considered in \cite{MM4} consists in
recovering the Jacobi matrix (i.e. the sequences
$\{a_1,a_2,\ldots\}$, $\{b_1, b_2,\ldots \}$) and $a_0$ from the
response operator.

In what follows we use the same notations for operators and for
matrices of these operators.

We introduce the \emph{inner space} of dynamical system
(\ref{Jacobi_dyn}) $\mathcal{H}^T:=\mathbb{C}^T$, $h\in
\mathcal{H}^T$, $h=(h_1,\ldots, h_T)$ with the inner product
$h,l\in \mathcal{H}^T$, $(h,g)_{\mathcal{H}^T}=\sum_{k=1}^T
h_k\overline{g_k}$. The \emph{control operator}
$W^T:\mathcal{F}^T\mapsto \mathcal{H}^T$ is defined by the rule
\begin{equation}
\label{Control_op}
W^Tf:=u^f_{n,T},\quad n=1,\ldots,T.
\end{equation}
From (\ref{Jac_sol_rep}) we deduce the representation for $W^T$:
\begin{equation*}
\left(W^Tf\right)_n=u^f_{n,T}=\prod_{k=0}^{n-1}
a_kf_{T-n}+\sum_{s=n}^{T-1}w_{n,s}f_{T-s-1},\quad n=1,\ldots,T.
\end{equation*}
The following statement is equivalent to boundary controllability
of (\ref{Jacobi_dyn}).
\begin{lemma}
The operator $W^T$ is an isomorphism between
$\mathcal{F}^T$ and $\mathcal{H}^T$.
\end{lemma}

Along with the system (\ref{Jacobi_dyn}) we consider an auxiliary
system associated with the complex conjugate matrix $\overline A$:
\begin{equation}
\label{Jacobi_dyn_aux}
\begin{cases}
v_{n,t+1}+v_{n,t-1}-\overline{a_{n}}v_{n+1,t}-\overline{a_{n-1}}v_{n-1,t}-\overline{b_n}v_{n,t}=0,\quad n,t\in \mathbb{N},\\
v_{n,-1}=v_{n,0}=0,\quad n\in \mathbb{N}, \\
v_{0,t}=f_t,\quad t\in \mathbb{N}\cup\{0\}.
\end{cases}
\end{equation}
The objects corresponding to the system (\ref{Jacobi_dyn_aux}) are
marked with the symbol $\#$. Direct calculations show:
\begin{lemma}
The control and response operators of the system $\#$ a related
with control and response operators of the original system by the
relations
\begin{equation}
\label{adj_prop}
W^T_{\#}=\overline{W^T},\quad
R^T_{\#}=\overline{R^T},
\end{equation}
that is, the matrix of $W^T_{\#}$ and the response vector $r_{\#}$
are complex conjugate of the matrix of $W^T$ and the vector $r$.
\end{lemma}

For systems (\ref{Jacobi_dyn}), (\ref{Jacobi_dyn_aux}) we
introduce the \emph{connecting operator} $C^T:
\mathcal{F}^T\mapsto \mathcal{F}^T$ via the bilinear form: for
arbitrary $f,g\in \mathcal{F}^T$ we define
\begin{equation}
\label{C_T_def} \left(C^T
f,g\right)_{\mathcal{F}^T}=\left(u^f_{\cdot,T},
v^g_{\cdot,T}\right)_{\mathcal{H}^T}=\left(W^Tf,W^T_\#g\right)_{\mathcal{H}^T}.
\end{equation}
The following statement is crucial for solving the dynamic inverse
problem:
\begin{theorem}
The connecting operator $C^T$ is an isomorphism in
$\mathcal{F}^T$, it admits the representation in terms of inverse
data:
\begin{equation}
\label{C_T_repr} C^T=a_0C^T_{ij},\quad
C^T_{ij}=\sum_{k=0}^{T-\max{i,j}}r_{|i-j|+2k},\quad r_0=a_0,
\end{equation}
\begin{equation*}
C^T=
\begin{pmatrix}
r_0+r_2+\ldots+r_{2T-2} & r_1+r_3+\ldots+r_{2T-3} & \cdot &
r_T+r_{T-2} &
r_{T-1}\\
r_1+r_3+\ldots+r_{2T-3} & r_0+r_2+\ldots+r_{2T-4} & \cdot & \ldots
&r_{T-2}\\
\cdot & \cdot & \cdot & \cdot & \cdot \\
r_{T-3}+r_{T-1}+r_{T+1} &\ldots & \cdot & r_1+r_3 & r_2\\
r_{T}+r_{T-2}&\ldots & \cdot &r_0+r_2&r_1 \\
r_{T-1}& r_{T-2}& \cdot & r_1 &r_0
\end{pmatrix}.
\end{equation*}
\end{theorem}

The relations (\ref{adj_prop}) imply the following
\begin{remark}
The connecting operator is complex symmetric:
\begin{equation*}
\left(C^T\right)^*=\overline{C^T}, \quad \text{or} \quad
\left(C^T\right)^\top=C^T.
\end{equation*}
\end{remark}

\subsection{Inverse problem. }

Due to the finite speed of wave propagation in (\ref{Jacobi_dyn})
the solution $u^f$ depends on the coefficients $a_n,b_n$ as
follows:
\begin{remark}
\label{Rem1} For $M\in \mathbb{N}$, $u^f_{M-1,M}$ depends on
$\{a_0,\ldots,a_{M-1}\}$, $\{b_1,\ldots,b_{M-1}\}$, the response
$R^{2T}$ (or, what is equivalent, the response vector
$(r_0,r_1,\ldots,r_{2T-2})$) depends on $\{a_0,\ldots,a_{T-1}\}$,
$\{b_1,\ldots,b_T\}$.
\end{remark}
Thus the natural set up of the dynamic IP for (\ref{Jacobi_dyn}):
by the given operator $R^{2T}$ to recover $\{a_0,\ldots,a_{T-1}\}$
and $\{b_1,\ldots,b_{T-1}\}$.

We also note that $a_0=r_0$, which follows from (\ref{R_def}).

In \cite{MM4} the authors proposed two methods for recovering the
coefficients $(a_k)^2$, $b_k,$ $k=1,\ldots$. To recover $a_k$ it
is necessary to use additional information, such as the sequence
of signs. Note that the results obtained for dynamic inverse data
in \cite{MM4} corresponds to results obtained for spectral inverse
data in \cite{G}.

Note that the impossibility of recovering $a_k$ is not the weak
point of the method, but a feature of the problem:
\begin{theorem}
\begin{itemize}
\item[1)] For $f\in \mathcal{F}$ the value $u^f_{n,t}$ is odd with
respect to $a_1,a_2,\ldots a_{n-1}$ and even with respect to
$a_n,a_{n+1},\ldots$.

\item[2)] The response vector depends on $\left(a_1\right)^2,
\left(a_2\right)^2,\ldots$
\end{itemize}
\end{theorem}

Now we set up a question: can one determine whether a vector
$(r_0,r_1,r_2,\ldots,r_{2T-2})$ is a response vector for dynamical
system (\ref{Jacobi_dyn}) with some $(a_0,\ldots,a_{T-1})$
$(b_1,\ldots,b_{T-1})$? The answer is the following theorem.
\begin{theorem}
\label{Th_char} The  vector $(r_0,r_1,r_2,\ldots,r_{2T-2})$ is a
response vector for the dynamical system (\ref{Jacobi_dyn}) if and
only if the complex symmetric matrix $C^{T-k}$, $k=0,1,\ldots,T-1$
constructed by (\ref{C_T_repr}) is an isomorphism in
$\mathcal{F}^{T-k}$.
\end{theorem}

\section{Complex moment problem and dynamic inverse problem. }

In what follows we will assume that additional parameter $a_0=1$.

With the dynamical system (\ref{Jacobi_dyn_int}) one can also
associate the control, response and connecting operators $W^T_N,$
$R^T_N$ and $C^T_N$ with the same formulas (\ref{Control_op}),
(\ref{Resp_op}), (\ref{C_T_def}), where instead of $u^f$ one
should use $v^f$.

Remark \ref{Rem1} in particular implies that
\begin{eqnarray}
\label{R_eqv} R^{2N-2} = R^{2N-2}_{N},\\
u^f_{n,\,t}=v^f_{n,\,t},\quad n\leqslant t\leqslant
N,\notag\\
W^T=W^T_{N}, \quad C^T=C^T_N, \quad  T\leqslant N.\notag
\end{eqnarray}

By choosing the special control $f=\delta=(1,0,0,\ldots)$, the
kernel of a response operator can be determined as (cf.
(\ref{con11})):
\begin{equation*}
\left(R^T_N\delta\right)_t=v^\delta_{1,t}= r^N_{t-1},\quad
t=1,2,\ldots.
\end{equation*}
Note that from (\ref{R_eqv}) it follows that
\begin{equation*}
r_t=r_t^N,\quad t=0,1,\ldots,2N-2.
\end{equation*}
Thus for special control $f=\delta$, using (\ref{Repr2}),
(\ref{CT_coeff}), (\ref{T_coeff}) one have that:
\begin{equation}
\label{Repr6} r_{t-1}^N=v^\delta_{1,t}= \sum_{k=1}^N
c^k_t\overline {u^k_1}=\sum_{k=1}^N
c^k_t=\sum_{k=1}^N\frac{1}{\rho_k}T_t(\omega_k),\quad
t=1,2,\ldots,
\end{equation}
where $\rho_k$ and $\omega_k$ are defined in (\ref{Rho}) and
(\ref{Omega}).

We introduce a discrete measure $d\rho^N$ on $\mathbb{C}$,
concentrated on the set of points
$\left\{\omega_k\right\}_{k=1}^N$, by definition we set
\begin{equation}
\label{MeasureR} d\rho^N(\{\omega_k\})=\frac{1}{\rho_k},
\end{equation}
so that at points $\omega_k$ it has weights $\frac{1}{\rho_k}$.
Then we can rewrite (\ref{Repr6}) in a form which resembles the
spectral representation of dynamic inverse data (see
\cite{MM2,MM3,MM5}):
\begin{proposition}
The dynamic response vector of the system (\ref{Jacobi_dyn_int})
admits the following representation:
\begin{equation}
\label{Resp_vect_N}
r_{t-1}^N=\int_{\mathbb{C}}T_t(\lambda)\,d\rho^N(\lambda),\quad
t=1,2,\ldots.
\end{equation}
\end{proposition}


With a set of moments (\ref{Moment_eq}) we associate the following
Hankel matrices:
\begin{equation}
\label{Hankel}
S^n:=\begin{pmatrix} s_{2n-2} & s_{2n-3} & \ldots & s_{n-1}\\
s_{2n-3} & \ldots & \ldots & \ldots\\
\cdot & \cdot & \ldots & s_{1} \\
s_{n-1} & \ldots & s_{1} & s_{0}
\end{pmatrix},\quad n=2,3,\ldots.
\end{equation}
We also introduce the matrix $J_n\in \mathbb{C}^{n\times n}$
\begin{equation*}
J_n=\begin{pmatrix}
0 & 0 & 0 & \ldots & 1\\
0 & 0 & 0 & \ldots  & 0\\
\cdot & \cdot & \cdot & \cdot &  \cdot \\
0 & \ldots & 1& 0  & 0\\
\cdot & \cdot & \cdot & \cdot &  \cdot \\
1 & 0 & 0 & 0 &  0
\end{pmatrix},\quad n=2,3,\ldots.
\end{equation*}

In \cite{MM3,MM5} the authors obtained the following
\begin{proposition}
\label{Prop_resp_moments} The elements of the response vector
(\ref{Resp_vect_N}) of the system (\ref{Jacobi_dyn_int}) are
related to the moments
$s_k=\int_{\mathbb{C}}\lambda^k\,d\rho^N(\lambda)$ by the
following rule:
\begin{equation}
\label{Resp_mom_relat}
\begin{pmatrix}
r_0^N\\
r_1^N\\
\ldots \\
r_{n-1}^N
\end{pmatrix}=\Lambda_n\begin{pmatrix}
s_0\\
s_1\\
\ldots \\
s_{n-1}
\end{pmatrix},\quad n=1,2,\ldots,2N-1.
\end{equation}
where the entries of the matrix $\Lambda_n\in \mathbb{R}^{n\times
n}$ are given by
\begin{equation*}
\{\Lambda_n\}_{ij}=a_{ij}=\begin{cases} 0,\quad \text{if $i<j$},\\
0,\quad \text{if $i+j$ is odd,}\\
E_{\frac{i+j}{2}}^j(-1)^{\frac{i+j}{2}+j},\quad \text{otherwise}
\end{cases}
\end{equation*}
where $E_n^k$ are binomial coefficients.

The following relation holds:
\begin{equation*}
C^N=\widetilde\Lambda_N S^N
\left(\widetilde\Lambda_N\right)^*,\quad
\widetilde\Lambda_N:=J_N\Lambda_N J_N.
\end{equation*}

\end{proposition}

Thus, we arrived at the following procedure for solving the
truncated complex moment problem, i.e. the problem of finding a
measure that satisfies a finite number of moment equalities
(\ref{Moment_eq}) for $s_0,s_1,\ldots s_{2N-2}\in \mathbb{C}$:
\begin{procedure}
\label{Proc} The solution to the truncated moment problem can be
constructed by performing the following steps
\begin{itemize}
\item[1)] Go from $s_0,s_1,\ldots s_{2N-2}$ to $r_0,\ldots,
r_{2N-2}$ using the formula (\ref{Resp_mom_relat}).

\item[2)] Check if $r_0,r_1,\ldots r_{2N-2}$ satisfy the
conditions of Theorem \ref{Th_char}, then the elements of the
matrix $A^N$ can be calculated using the formulas from \cite{MM4}.

\item[3)] Calculate parameters $d_i$, $\rho_i$, $H_{ki}$,
$\omega_i$ using factorization (\ref{Factor}) and formulas
(\ref{EigVal}), (\ref{Rho}), (\ref{HKL}), (\ref{Omega}).

\item[4)] Construct a measure that solves the truncated moment
problem by (\ref{MeasureR})

\end{itemize}
\end{procedure}

The possibility of constructing a solution to the truncated
complex moment problem, described above, can be used as a basis
for the following
\begin{theorem}
If the moments $s_0,s_1,\ldots$ are such that the Hankel matrix
$S^T$ (\ref{Hankel}) is non-singular for every $T\in N$, there
exist a solution to the complex moment problem (\ref{Moment_eq}).
\end{theorem}
\begin{proof}
Taking an arbitrary $T\in \mathbb{N}$, and using the fact that for
$T'\leqslant T$ the corresponding $C^{T'}$ and $S^{T'}$ are
simultaneously non-singular, we use Procedure \ref{Proc} and
obtain the measure $d\rho^T$ that solves the truncated moment
problem for the set of moments $s_0,s_1,\ldots,s_{2T-2}$.

Then for arbitrary polynomial $P\in C[\lambda]$,
$P(\lambda)=\sum_{k=0}^Ma_k\lambda^k$, we have that
\begin{equation}
\label{measure_conv}
\int_{\mathbb{C}}P(\lambda)\,d\rho^N(\lambda)\longrightarrow_{N\to\infty}
\sum_{k=0}^Ma_ks_k.
\end{equation}
this follows from the fact that by the construction we have that
\begin{equation*}
\int_{\mathbb{C}}P(\lambda)\,d\rho^N(\lambda)=\sum_{k=0}^Ma_ks_k,\quad
\text{if} \quad 2N-2>M.
\end{equation*}
Convergence (\ref{measure_conv}) means that $d\rho^N$ converges
$*-$weakly as $N\to\infty$ to some measure $d\rho$, and this
$d\rho$ is a solution to the complex moment problem.
(\ref{Moment_eq}).

\end{proof}

Note that our method allows one to associate a certain measure
with the Jacobi matrix (\ref{Jac_matr}). The study of this
measure, its connection with the generalized spectral function
\cite{G} as well as related functions \cite{MMS} will be the
subject of the forthcoming publications.

\end{document}